\documentclass{amsart}
\usepackage{amsmath}
\usepackage{amsthm}
\usepackage{amsfonts}
\newtheorem{theorem}{Theorem}
\newtheorem{lemma}[theorem]{Lemma}
\newtheorem{cor}[theorem]{Corollary}

\usepackage{url}
\DeclareMathOperator\lcm{lcm}

\title[Lower Bounds  for Least Common Multiples of Arithmetic Progressions]{Further Improvements of Lower Bounds
\\for the Least Common Multiples
\\of Arithmetic Progressions}

\author{Shaofang Hong}
\address{Mathematical College, Sichuan University\newline\indent
    Chengdu 610064\newline\indent People's Republic of China}
\email{s-f.hong@tom.com, hongsf02@yahoo.com, sfhong@scu.edu.cn}

\author[Scott Duke Kominers]{Scott Duke Kominers$^*$}
\address{Departments of Mathematics and Economics, Harvard University\newline\indent c/o 8520
  Burning Tree Road\newline \indent Bethesda, MD 20817, USA}
\email{kominers@fas.harvard.edu, skominers@gmail.com}
\urladdr{http://www.scottkom.com/}
\thanks{$^*$Corresponding author.}

\thanks{The first author was partly supported by the Program for New Century Excellent Talents in
University, Grant No.~NCET-06-0785.  The second author was partly supported by a National Science Foundation Graduate Research Fellowship.}

\subjclass[2000]{11A05 (primary)}
\keywords{Least common multiple, arithmetic progression}

\begin{document}
\begin{abstract}For relatively prime positive integers $u_0$ and $r$, we consider the
arithmetic progression $\{u_k:=u_0+kr\}_{k=0}^n$.

Define $L_n:=\lcm\{u_0, u_1, \ldots, u_n\}$ and let $a\ge 2$ be
any integer. In this paper, we show that, for integers $\alpha ,
r\geq a$ and $n\geq 2\alpha r$, we have $$L_n\geq u_0r^{\alpha
+a-2}(r+1)^n.$$ In particular, letting $a=2$ yields an improvement
to the best previous lower bound on $L_n$ (obtained by Hong and
Yang) for all but three choices of $\alpha , r\geq 2$.
\end{abstract}
\maketitle
\section{Introduction}\label{intro}

The search for effective estimates on the least common multiples of
finite arithmetic progressions began with the work of
Hanson~\cite{Ha} and Nair~\cite{Na}, who respectively found upper
and lower bounds for $\lcm\{1,\ldots, n\}$.

Inspired by this work, Bateman, Kalb, and Stenger~\cite{BKS} and
Farhi~\cite{BF1} respectively sought asymptotics and nontrivial
lower bounds for the least common multiples of general arithmetic
progressions.  Farhi~\cite{BF1} obtained several nontrivial bounds
and posed a conjecture which was later confirmed by Hong and
Feng~\cite{HF}.  Additionally, Hong and Feng~\cite{HF} obtained an
improved lower bound for sufficiently large arithmetic progressions;
this result was recently sharpened further by Hong and
Yang~\cite{HY}.  Hong and Yang~\cite{HY2} and Farhi and
Kane~\cite{FK} also obtained some related results regarding the
least common multiple of a finite number of consecutive integers.
The theorem of Farhi and Kane~\cite{FK} was very recently extended
to general arithmetic progressions by Hong and Qian~\cite{HQ}.

In this article, we study finite arithmetic progressions
$\{u_k:=u_0+kr\}_{k=0}^n$ with $u_0,r\geq 1$ integers satisfying
$(u_0,r)=1$.  Throughout, we define
$$L_n:=\lcm\{u_0,u_1,\ldots, u_n\}$$
to be the least common multiple of the sequence $\{u_k\}_{k=0}^n$.
The following lower bound on $L_n$ was found by Hong and
Yang~\cite{HY}.
\begin{theorem}[\cite{HY}]\label{hythm}
Let $\alpha\geq 1$ be an integer.  If $n>  r^\alpha$, then we have
$L_n\geq u_0 r^\alpha(r+1)^n$.
\end{theorem}

\noindent If $r=1$, then the content of Theorem~\ref{hythm} is the
conjecture of Farhi~\cite{BF1} proven by Hong and Feng~\cite{HF}. If
$\alpha =1$, then Theorem~\ref{hythm} becomes the improved lower
bound of Hong and Feng~\cite{HF}.

In this paper, we sharpen the lower bound in Theorem~\ref{hythm}
whenever $\alpha , r\ge 2$. In particular, we prove the following
theorem which replaces the exponential condition $n>r^{\alpha }$ of
Theorem~\ref{hythm} with a linear condition, $n\ge 2\alpha r$.

\begin{theorem}\label{mainthm}
Let $a\ge 2$ be any given integer. Then for any integers $\alpha ,
r\ge a$ and $n\ge 2\alpha r$, we have $L_n\ge u_0r^{\alpha
+a-2}(r+1)^n$.
\end{theorem}

Letting $a=2$, we see that Theorem~\ref{mainthm} improves upon
Theorem~\ref{hythm} for all but three choices of $\alpha , r\ge 2$.

The remainder of this paper is organized as follows. In Section
\ref{sec2}, we introduce relevant notation and previous results.  In
Section~\ref{sec3}, we prove Theorem~\ref{mainthm} and as a
corollary obtain arbitrarily strong sharpening of
Theorem~\ref{hythm} which apply in all but finitely many cases.
Then, in Section~\ref{sec4}, we discuss when the condition
$n>r^\alpha$ is necessary in Theorem~\ref{hythm}.

\section{Notation and Previous Results}\label{sec2}  For any real numbers~$x$ and~$y$,
we say that~$y$ \emph{divides}\/~$x$ if there exists an integer~$z$
such that $x=y\cdot z$. If $x$ divides $y$, then we write $y\mid x$.
As usual, we let $\lfloor x \rfloor$ denote the largest integer no
more than $x$.

Following Hong and Yang~\cite{HY}, we denote, for each integer
$0\leq k\leq n$,
$$C_{n,k}:=\frac{u_k\cdots u_n}{(n-k)!},\quad L_{n,k}:=\lcm\{u_k,\ldots,u_n\}.$$
From the latter definition, we have that $L_n=L_{n,0}$.

The following Lemma first appeared in~\cite{BF1} and was reproven in
several sources:
\begin{lemma}[\cite{BF1}, \cite{BF2}, \cite{HF}]\label{L1}
For any integer $n\geq 1$, $C_{n,0}\mid L_n$.
\end{lemma}
From Lemma \ref{L1}, we see immediately that
\begin{equation}
\label{keq}L_{n,k}=A_{n,k}\frac{u_k\cdots u_n}{(n-k)!}=A_{n,k}\cdot
C_{n,k}
\end{equation}
for an integer $A_{n,k}\geq 1$.

Following Hong and Feng~\cite{HF} and Hong and Yang~\cite{HY}, we
define, for any $n\geq 1$,
\begin{equation}
k_n:=\max\left\{0,\left\lfloor
\frac{n-u_0}{r+1}\right\rfloor+1\right\}.\label{hat}
\end{equation}
Hong and Feng~\cite{HF}  proved the following result.
\begin{lemma}[\cite{HF}]\label{L2}
For all $n\geq 1$ and $0\leq k\leq n$,
$$L_n\geq L_{n,k_n}\geq C_{n,k_n}\geq  u_0(r+1)^n.$$
\end{lemma}

\section{Proof of the Main Theorem and Corollary}\label{sec3}
We begin with a lemma which is similar to a key step of the proof of
Theorem~\ref{hythm}. The proof of this result closely follows the
approach of Hong and Yang~\cite{HY}, but simplifies the analysis.

\begin{lemma}\label{lemma2}
Let $a\ge 2$ be any given integer. Then for any integers $\alpha ,
r\ge a$ and $n\ge 2\alpha r$, we have $n-k_n>(\alpha +a-2)r$.
\end{lemma}
\begin{proof}
If $n\le u_0$, then by the definition (2) we have $k_n\le 1$. Since
$\alpha , r\ge a\ge 2$ and $n\ge 2\alpha r$, we deduce that
$n-k_n\ge n-1\ge 2\alpha r-1>(\alpha +a-2)r$.

Now, we suppose that $n>u_0$. In this case, we have
$$k_n=\left\lfloor \frac{n-u_0}{r+1} \right\rfloor+1;$$ it follows
that
$$
k_n\le \frac{n-u_0}{r+1}+1\le \frac{n-1}{r+1}+1=\frac{n+r}{r+1}.
$$
From this, we then see that
\begin{equation}\label{part1}
n-k_n\ge n-\frac{n+r}{r+1}=\frac{(n-1)r}{r+1}\ge \frac{(2\alpha
r-1)r}{r+1}.
\end{equation}

However, the assumption $\alpha , r\ge a$ implies that
\begin{align}
(2\alpha r-1)-(r+1)(\alpha +a-2)\nonumber&=(r-1)\alpha -1-(r+1)(a-2)
\\&\ge a(r-1)-1-(r+1)(a-2)\label{part2}\\\nonumber&=2(r-a)+1>0.
\end{align}
Therefore from \eqref{part2}, we infer that
\begin{equation}\label{part3}
\frac{2\alpha r-1}{r+1}>\alpha +a-2.
\end{equation}
The desired result then follows immediately from \eqref{part1} and
\eqref{part3}.
\end{proof}

From Lemma~\ref{lemma2}, the proof of Theorem~\ref{mainthm} follows
directly, via the same argument as in the endgame of the proof of
Theorem~\ref{hythm}. For completeness, we reproduce this elegant
argument here.

\begin{proof}[Proof of Theorem~\ref{mainthm}]
By hypothesis, we have $\alpha , r\ge a\ge 2$ and $n\ge 2\alpha r$.
As a consequence of Lemma \ref{lemma2}, we therefore obtain that
$r^{\alpha +a-2} \mid (n-k_n)!$.  Thus, we may express $(n-k_n)!$ in
the form $r^{\alpha +a-2} \cdot B_n=(n-k_n)!$, with $B_n\geq 1$ an
integer. If we choose $k=k_n$ in \eqref{keq}, we find that
$$r^{\alpha +a-2} \cdot B_n\cdot L_{n,k_n}=A_{n,k_n}\cdot u_{k_n}\cdots u_n.$$
It then follows that $r^{\alpha +a-2} \mid A_{n,k_n}$, since the
requirement $(r,u_0)=1$ implies that $(r,u_k)=1$ for all $0\leq
k\leq n$.  Then, we obtain from \eqref{keq} and Lemma \ref{L2} that
$$L_{n,k_n}\geq r^{\alpha +a-2} C_{n,k_n}\geq
u_0r^{\alpha}(r+1)^n;$$ Theorem~\ref{mainthm} follows.
\end{proof}

As a corollary of Theorem~\ref{mainthm}, we obtain a substantial
sharpening of Theorem~\ref{hythm}.
\begin{cor}\label{theCor}
Fix integers $a\ge 2$ and $\beta \ge 1$. Then, for all but finitely
many choices of integers $\alpha , r\ge a$, we have that $L_n\ge
u_0r^{\alpha +\beta +a-2}(r+1)^n$ whenever $n>r^{\alpha }$.
\end{cor}
\begin{proof}
By Theorem~\ref{mainthm}, we have $L_n\ge u_0r^{\alpha +\beta
+a-2}(r+1)^n$ whenever $n\ge 2(\alpha +\beta +a-2)r$. If $r^{\alpha
}+1\ge 2(\alpha +\beta +a-2)r$, then the condition $n>r^{\alpha }$
guarantees that $n\ge 2(\alpha +\beta +a-2)r$. Since, for any given
integer $\beta \ge 1$, we have $r^{\alpha }+1\ge 2(\alpha +\beta
+a-2)r$ for all but finitely many choices of $\alpha , r\ge a$, the
result follows immediately.
\end{proof}

The bound of Corollary~\ref{theCor} becomes effective even for
small~$\alpha $ and~$r$. For example, the choices of $a=2$ and
$\beta =1$ in Corollary~\ref{theCor} sharpen  Theorem~\ref{hythm} by
a factor of $r$ for all but six choices of $\alpha ,r\ge 2$.

\section{Examples with $L_n<u_0r^{\alpha }(r+1)^n$}\label{sec4}
In their article, Hong and Yang~\cite{HY} asserted that their
condition $n>r^\alpha$ is actually necessary for the bound
$L_n>u_0r^\alpha(r+1)^n$ in Theorem \ref{hythm}. This assertion was
accompanied by an example,
\begin{equation}
u_0=r=2,\quad \alpha=3,\quad n=8,\label{exam}
\end{equation}
in which $L_n=5040< 104976=u_0r^\alpha(r+1)^n$ (see Remark 3.1
of~\cite{HY}). This example \eqref{exam} not only satisfies
$r^\alpha=8\not<8=n$, but also satisfies $2\alpha r=12\not\leq 8=n$.
Unfortunately, \eqref{exam} does not satisfy the condition
$(u_0,r)=1$, so it does not actually suffice to demonstrate the
necessity of the condition $n>r^\alpha$ in Theorem \ref{hythm} when
$r=2$ and $\alpha=3$.

As $2\alpha r<r^{\alpha }+1$ for all but three choices of $\alpha ,
r\ge 2$, examples with $L_n<u_0r^\alpha(r+1)^n$ and $n=r^\alpha$ are
available for at most three choices of $\alpha,r\geq 2$. A computer
search of all $u_0< n=r^\alpha$ with $(u_0,r)=1$ in these three
cases\footnote{We need only consider the cases with $u_0< n$, as the
proof of Lemma~\ref{lemma2} shows that $\alpha r<n-k_n$ \emph{a
priori}---and so the result of Theorem~\ref{mainthm}
holds---whenever $u_0\geq n$.} indicates that there exists only one
example with  $L_n<u_0r^\alpha(r+1)^n$, $(u_0,r)=1$, and
$r^\alpha=n$:
\begin{gather*}u_0=1,\quad r=\alpha=2, \quad n=4,
\end{gather*}
 in which $L_{4}=\lcm\{1,3,5,7,9\}=315<324=1\cdot 2^2(2+1)^4$.

\section*{Acknowledgements}
The authors greatly appreciate the helpful comments and suggestions
of Andrea~J.\ Hawksley, Professor Noam~D.\ Elkies, Daniel M.\ Kane, Shrenik
N.\ Shah, and the editor, Professor Wen-Ching Winnie Li.

\bibliographystyle{amsalpha}
\bibliography{lcm_bound_bib}

\end{document}